\documentclass[11pt]{amsproc}
\usepackage{amsfonts}
\usepackage{amsmath,amstext,amsbsy,amssymb}

\newtheorem{theorem}{Theorem}[section]
\newtheorem{lemma}[theorem]{Lemma}

\newtheorem{corollary}[theorem]{Corollary}

\theoremstyle{definition}

\theoremstyle{remark}
\newtheorem{remark}[theorem]{Remark}

\numberwithin{equation}{section}

\newcommand{\la}{\lambda}
\newcommand{\al}{\alpha}
\newcommand{\ep}{\epsilon}
\begin{document}

\title{A generalization of Newton's identity and Macdonald functions}
\author{Wuxing Cai}
\address{School of Sciences,
South China University of Technology, Guangzhou 510640, China}
\email{caiwx@scut.edu.cn}
\author{Naihuan Jing*}
\address{Department of Mathematics,
   North Carolina State University,
   Raleigh, NC 27695-8205, USA}
\email{jing@math.ncsu.edu}
\thanks{*Corresponding author: Naihuan Jing}

\begin{abstract}
 A generalization of Newton's identity on symmetric functions is given. Using the generalized Newton identity we
  give a unified method to show the existence of Hall-Littlewood, Jack and Macdonald polynomials. We also give
  a simple proof of the Jing-J\"ozefiak formula for two-row Macdonald functions.
\end{abstract}
\maketitle
\section{Introduction}
Originally, one form of Newton's identity on symmetric polynomials is $ke_k=\sum_{i=1}^k(-1)^{i-1}p_ie_{k-i}$, where $p_i$ and $ e_i$ are power sum symmetric polynomials and elementary symmetric polynomials respectively in finite but arbitrary variables. As this number goes to infinity, the identity becomes an identity in the ring of
 symmetric functions. Applying the standard involution $\omega$ which sends $p_i$ to $(-1)^{i-1}p_i$ and $e_i$ to $s_i$ on this identity, Newton's identity becomes
\begin{equation}\label{F:NewtonIdentity}
ks_k=\sum_{i\geq1}p_is_{k-i},
\end{equation}
where $s_k$ is the one-row Schur function or the $k$th complete symmetric function. Applications of these identities can be found in Galois theory, invariant theory, group theory, combinatorics and so on. Concerning (\ref{F:NewtonIdentity}), one may wonder what would happen to $\sum_{i,j\geq1}p_{i+j}s_{m-i}s_{n-j}$ for a two-row partition $\la=(m,n)$ or generally to $A=\sum_{i_1,\cdots,i_s\geq1}p_{i_1+\cdots+i_s}s_{\la_1-i_1}\cdots s_{\la_s-i_s}$ for a partition $\la$ of length $s$. It is clear that $A$ is homogeneous of degree $|\la|$, thus it should be of the form $\sum_{\mu\vdash|\la|} c_{\la\mu}q_{\mu}$, where $q_\mu=s_{\mu_1}s_{\mu_2}\cdots$. In fact, a corollary of our generalized Newton's formula asserts that $c_{\la\mu}\neq0$ only if $\mu\geq\la$; (see the remark of Corollary \ref{R:directgeneralizedNewton}). For $\la=(m,n)$ the identity is given by Lemma 3.8 in \cite{CJ2}, where it plays a crucial role in the proof of the Laplace-Beltrami operator for Jack functions.

Jack symmetric functions \cite{St} are certain one-parameter deformation of Schur symmetric functions. They are special cases of the more general Macdonald symmetric functions \cite{M} in two free parameters. When the two parameters
take special forms they include a variety of classical symmetric functions including Jack functions and Hall-Littlewood functions.

Two driving problems of Macdonald polynomials are the inner product conjecture and
the positivity conjecture \cite{M}. Cherednik introduced the double affine Hecke algebras (DAHA) to solve the
inner product conjecture by extending the problem to the more general nonsymmetric Macdonald polynomials \cite{C}.
Haiman's proof \cite{Ha} of Garsia-Haiman's $n!$ conjecture (which implies Macdonald positivity conjecture) demonstrated the broadness of Macdonald polynomials and their deep connection with Hilbert schemes and resolution
of singularities. Similar positivity phenomena have also been observed in
Lascoux-Lapointe-Morse $k$-Schur functions \cite{LLM}. Haglund's combinatorial formula for Macdonald polynomials
also relied on more general combinatorial statistics \cite{Hg, HHL} (also \cite{As}). These developments often take place by
generalizing Macdonald polynomials.

On the other hand, Macdonald polynomials have also been studied by exploring different characterizations
and related applications.
The interpolation Macdonald polynomials of Okounkov \cite{O},
Knop and Sahi \cite{KS} (see also Rains \cite{R}, Coskun
and Gustafson \cite{CG}) showed that Macdonald polynomials can be characterized by
the zero patterns. This part of developments demonstrated the importance of Macdonald polynomials in modern
harmonic analysis and related fields. For example, Lascoux gave a new characterization of
Macdonald polynomials by using the Yang-Baxter equation and difference operators \cite{La}.
Interesting connections with Selberg-type integrals have also been studied in
the context of Macdonald polynomials, which have
triggered new studies on more generalized hypergeometric series (see \cite{Ka, W, Sc}).
The most general form of Macdonald type hypergeometric functions has also been studied
by Schlosser \cite{Sc}.

 In this paper, we seek a new characterization and generalization of Macdonald polynomials by using a commutative algebra generalization of Newton's identity in symmetric functions. Based on the generalized Newton formula
 we construct a self-adjoint vertex operator which acts as certain raising operators
  on generalized complete symmetric functions. The eigenvectors of this operator are essentially Macdonald symmetric functions. In fact, two specializations to the operator leads to two different operators; both of the operators diagonalize the original Macdonald functions. One of the operators was considered by Shiraishi (cf. \cite{S}). Applying this method, we hope that a further generalization of Newton's identity or a generalized vertex operator would lead to some generalized Macdonald symmetric functions.

 In Macdonald's original construction \cite{M}, a discrete Casimir operator was used to diagonalize Macdonald functions in finitely many variables; and then passed on to infinitely many variables. The operator considered in this paper
 will directly ensure the triangularity property and thus the existence follows immediately; this can be viewed as
 a generalization of our previous work \cite{CJ2}. As specialization of the differential operator, we also obtain
 an operator that diagonalizes Hall-Littlewood polynomials directly. This suggests that one may develop a theory of Hall-Littlewood polynomials as some eigenfunctions of a differential operator
 directly without reference to Macdonald polynomials. It would be interesting to see if this operator
 has any meaning over p-adic fields as Hall-Littlewood polynomials are spherical functions over $GL_n(\mathfrak p)$
 \cite{M}.

 In the development of classical symmetric functions, vertex operators have been helpful in realization of
 Hall-Littlewood symmetric functions \cite{J}
 as well as rectangular Jack polynomials \cite{CJ}.
Recently we have also derived an iterated vertex operator realization of Jack symmetric functions \cite{CJ3}.
We hope the new vertex operator constructed in this paper will play a role in vertex operator realization of
Macdonald polynomials. It is also interesting to note the similarity of our vertex operator with the operator
considered by Garsia-Haiman \cite{GH}.

This paper is organized as follows. In section 2, we introduce some definitions and recall some simple results about partitions and symmetric functions. In section 3, we generalize Newton's identity in a commutative algebra. In section 4, this formula is used to find a vertex operator that has Macdonald functions as eigenvectors, and we also recover the explicit formula for two-row Macdonald functions.

\section{Partitions and symmetric functions}
A partition $\la=(\la_1,\cdots,\la_s)$ of $n$ is a sequence of weakly decreasing non-negative integers,
and $n$ is the weight of $\la$: $n=|\la|=\sum_i\la_i$, which we also denote as $\la\vdash n$. Sometimes we also write $\la=(1^{m_1}2^{m_2}\cdots)$, where $m_i=m_i(\la)$ is the multiplicity of $i$ in $\la$. The length of $\la$ is defined as $l(\la)=\sum_im_i$ and we also denote $z_\la=1^{m_1}m_1!2^{m_2}m_2!\cdots\in\mathbb{N}$. One defines the union of two partitions $\la\cup\mu$ by $m_i(\la\cup\mu)=m_i(\la)+m_i(\mu)$ for all $i$. Let $\mathcal P_n$ be the set of partitions of weight $n$ and $\mathcal P$ the set of all partitions.
For $\la,\mu\in\mathcal P_n$, the dominance order $\la\geq\mu$ is defined by $\sum_{j\leq i}(\la_j-\mu_j)\geq0$ for all $i$.
We write $\la>\mu$ if $\la\geq\mu$ but $\la\neq\mu$.

The ring $\Lambda$ of symmetric functions in the variables $x_i$ is a $\mathbb Z$-module
with basis $m_{\la}$, $\la\in\mathcal{P}$, where for a partition $\la=(\la_1,\cdots,\la_s)$, the monomial
symmetric function $m_\la(x)=\sum_{\alpha}x_1^{\alpha_1}x_2^{\alpha_2}\cdots$, where $\alpha$ runs over
distinct permutations of $\lambda$.
For example $m_{(n)}=x_1^n+x_2^n+\cdots$, $n\geq1$, is the power sum $p_n$. For convenience,
we define $p_0=1$ and $p_n=0$ when $n<0$. The power sum symmetric functions $p_\la=p_{\la_1}p_{\la_2}\cdots$
form another basis of $\Lambda_F=\Lambda\otimes_\mathbb{Z}F$ for a field $F$ of characteristic $0$.
The $p_\la$'s with $|\la|=n$ span a subspace of $\Lambda_F$, denoted as $\Lambda_F^n$. An element in $\Lambda_F^n$ is said to be homogeneous of degree $n$.

Let $\ep=(\ep_1, \ep_2, \cdots)$ be a sequence of (finite or infinite) non-zero parameters.
We will denote the extension field $\mathbb{Q}(\epsilon_1,\epsilon_2,\cdots)$ of $\mathbb{Q}$ by $F^\ep$. We use the notation $\Lambda(\ep)$ for the algebra of symmetric functions $\Lambda\otimes_\mathbb{Z}F^\ep$ which has been associated with the scalar product given by
\begin{align} \label{def}
\langle p_{\la}, p_{\mu}\rangle=\delta_{\la
\mu}\epsilon_\la z_\lambda, ~~(\la, \mu \in \mathcal P)
\end{align}
where $\delta$ is the Kronecker symbol, and  $\epsilon_\la=\epsilon_{\la_1}\epsilon_{\la_2}\cdots$, with $\epsilon_0=1$.

Fix an total order in $\mathcal{P}_n$ which is compatible with the dominance order, then we can apply Gram-Schmidt process to the basis $m_\la$'s to get an orthogonal basis $P^\ep_\la$'s such that the leading coefficients
are $1$. In \cite{Ke}, Kerov called this basis generalized Hall-Littlewood (GHL) functions and  characterized this basis from the $\ep_n$'s. Note that the specializations of $\epsilon_n=1, \alpha, (1-t^n)^{-1}$ or $\frac{1-q^n}{1-t^n}$ corresponds to Schur, Jack, Hall-Littlewood or Macdonald functions respectively.

 For each $\la\in\mathcal{P}$, define the generalized complete symmetric function $q_\la^\epsilon=q_{\la_1}^{\epsilon} q_{\la_2}^{\epsilon}\cdots$, where $q^\ep_n$ is given by the generating function:
$$\exp\Big(\sum_{n\geq1}
\frac{z^n}{n\epsilon_n}p_n\Big)=\sum_n q_{n}^\epsilon z^n.$$
Thus $q^\ep_n=0$ for $n<0$ and for $n\geq0$, we have
$$q_{n}^\epsilon=\sum_{\la\vdash n}\frac{p_\la}{z_\la\ep_\la}.$$
From this expression we see that $q^\ep_\la$ is homogeneous of degree $|\la|$.
Similar to the cases of Jack functions and Macdonald functions, the $m_\la$'s and $q_\la$'s are dual to each other as shown in the following lemma. Now assume that $P^\ep_\la$'s and $Q^\ep_\la$'s are dual orthogonal bases of $\Lambda(\ep)$. If $P^\ep_\la$ is of the form $P^\ep_\la=\sum_{\mu\leq\la}C_{\la\mu}m_\mu$ with $C_{\la\la}\neq0$, then $Q^\ep_\mu$ has the form $Q^\ep_\mu=\sum_{\la\geq\mu}D_{\mu\la}q^\epsilon_\la$ with $D_{\mu\mu}\neq0$.

Similar to Macdonald symmetric functions (cf. \cite{M}) we also have
\begin{lemma}
We have $\langle m_\mu, q^\epsilon_\la\rangle=\delta_{\la\mu}$.
\end{lemma}
\begin{proof}
For a partition $\la$, denote $p'_\la(x)=\ep_\la^{-1}z_\la^{-1}p_\la(x)$, then $\langle p'_\la, p_\mu\rangle=\delta_{\la\mu}$.
Set $q_\la^\epsilon=\sum_\rho a_{\la\rho}p'_\rho$, $m_\mu=\sum_\rho b_{\mu\rho}p_\rho$, then we need to prove
\begin{equation}\label{E:qmrangle}
\langle q_\la^\epsilon, m_\mu\rangle=\sum_\rho a_{\la\rho}b_{\mu\rho}=\delta_{\la\mu}.
\end{equation}

 We denote $p'_n(x)=p'_{(n)}(x)$ and set $$F=\exp\Big(\sum_{n\geq1}p'_n(x)p_n(y)\Big).$$
On one hand we have
\begin{align*}
F&=\prod_{n\geq1}\sum_{i\geq0}\frac{1}{i!}p'_n(x)^ip_n(y)^i=\sum_\la p'_\la(x)p_\la(y).
\end{align*}
On the other hand we have
\begin{align*}
F&=\prod_{n\geq1}\exp\big(\sum_i p'_n(x)y_i^n\big)=\prod_{n\geq1}\sum_{i\geq0}q_n^\epsilon(x)y_i^n\\
&=\sum_\la q_\la^\epsilon(x)m_\la(y)=\sum_\la \sum_{\rho,\tau}a_{\la\rho}b_{\la\tau}p'_\rho(x)p_{\tau}(y)\\
&=\sum_{\rho,\tau}\sum_\la a_{\la\rho}b_{\la\tau}p'_\rho(x)p_{\tau}(y).
\end{align*}
Thus we have $\sum_{\la}a_{\la\rho}b_{\la\tau}=\delta_{\rho\tau}$, which is equivalent to equation (\ref{E:qmrangle}).
\end{proof}

We need some linear operators on $\Lambda(\ep)$.  Define $h_n\in \text{End}_{F^\ep}(\Lambda(\ep))$ by
\begin{align}
h_n.v&=n\epsilon_n\frac{\partial}{\partial p_{n}}v,\\
h_{-n}.v&=p_{n}v,\\
h_0.v&=v,
\end{align} where $n>0$ and $v\in \Lambda(\ep)$.
It can be verified that
\begin{equation}
  h_mh_n-h_nh_m=m\epsilon_m\delta_{m, -n}h_0.
  \end{equation}
Thus the Lie sub-algebra of $\text{End}_{F^\ep}(\Lambda(\ep))$ generated by the $h_n$'s is a Heisenberg algebra.

For an operator $A$ on $\Lambda(\ep)$, the conjugate $A^*$ of $A$ is defined by $\langle A.u,v\rangle=\langle u,A^*.v\rangle$ for all $u,v\in \Lambda(\ep)$. An operator $A$ is called {\it self-adjoint} if $A=A^*$. For a partition $\la$, if $A.q^\ep_\la$ is of the form $A.q^\ep_\la=\sum_{\mu\geq\la}a_{\la\mu}q^\ep_\la$, we say that $A$ {\it raises} $q^\ep_\la$.

The following result is immediate from definition.
\begin{lemma}\label{L:propertyofh_n}
We have following the properties for $h_n$'s:\\
(1) $h_{n}^*=h_{-n}$ for $n\in\mathbb{Z}$ ,\\
(2) $h_i.q^{\ep}_n=q^{\ep}_{n-i}$, for $n\in\mathbb{Z}, i>0$.
\end{lemma}

\section{A generalization of Newton's formula}

For non-negative integers $a_1,\cdots, a_s$, we can rearrange them in decreasing order to get a partition, denoted as $[a_1,\cdots,a_s]$.
For any partition $\la=(\la_1,\cdots,\la_s)$ with $s\geq2$ and $\la_s>0$, we define the partition $\la^+=[\la_1,\cdots,\la_{s-2},\la_{s-1}+1,\la_s-1]$.
Clearly $\la^+>\la$.

The following fact on unions of partitions is well-known.
For two partitions $\mu\geq\la$, and a partition $\nu$, we have $\mu\cup\nu\geq\la\cup\nu$.

 For a commutative algebra $\mathcal{A}$ over a field $F$, let $R_n, q_n\in\mathcal{A}$ ($n\geq0$), and $q_0=1$ is the unit element of $\mathcal{A}$. For a partition $\la=(\la_1,\cdots,\la_s)$, define $q_\la=q_{\la_1}q_{\la_2}\cdots q_{\la_s}\in\mathcal{A}$. For convenience, we set $q_n=0$ for $n<0$.

\begin{theorem}\label{T:Newtongeneralization}

 Assume that $\sum_{i\geq1}R_iq_{n-i}=c_nq_n$ with $c_n\in F$  for each $n\geq1$. Then for any partition $\la=(\la_1,\cdots,\la_s)$ of length $s$, we have \begin{equation}
 \sum_{i_1,\cdots,i_s\geq1}R_{i_1+\cdots+i_s}q_{\la_1-i_1}\cdots q_{\la_s-i_s}=\sum_{\mu\geq\la}c_{\la\mu}q_\mu,
 \end{equation}
 where $c_{\la\mu}\in\sum_{i\geq0}\mathbb{Z}c_{\mu_i}$,
 and specifically $c_{\la\la}=(-1)^{l(\la)-1}c_{\la_s}$.
\end{theorem}

Iteratively using $\sum_{i\geq1}R_iq_{n-i}=c_nq_n$, we can write $R_n$ as linear combinations of $q_\la$'s
with $\la\vdash n$. Thus it is clear that $T.q_\la=\sum_{\mu\vdash|\la|}c_{\la\mu}q_\mu$.
We need to show that $q_\mu$ appears only when $\mu\geq\la$. Before we prove the theorem,
we give the following corollary of this theorem which directly generalizes Newton's identity.
\begin{corollary}\label{R:directgeneralizedNewton}
In the commutative algebra $\Lambda(\ep)$, assume that $\la=(\la_1,\cdots,\la_s)$ is a partition of length $s$. Then we have $$\sum_{i_1,\cdots,i_s\geq1}\epsilon_{i_1+\cdots+i_s}^{-1}p_{i_1+\cdots+i_s}q^\epsilon_{\la_1-i_1}\cdots q^\epsilon_{\la_s-i_s}=\sum_{\mu\geq\la} c_{\la\mu} q_\mu^\epsilon$$ with $c_{\la\la}=(-1)^{s-1}\la_s$.
\end{corollary}

\begin{proof}
We only need to prove that $\sum_{i\geq1}\epsilon_i^{-1}p_iq^\epsilon_{n-i}=nq^\epsilon_n$ for each $n\geq1$, and then the corollary follows directly as a special case of Theorem \ref{T:Newtongeneralization}, where we set $\mathcal{A}=\Lambda(\ep)$, $R_n=\epsilon_n^{-1}p_n$, $q_n=q^\epsilon_n$ and $c_n=n$.

 Consider the action of the operator $A=\sum_{i\geq1}\ep_i^{-1}h_{-i}h_i=\sum_{i\geq1}ip_{i}\frac{\partial}{\partial p_i}$ acting on $q_n^\ep$. On one hand, Lemma \ref{L:propertyofh_n} implies $A.q_n^\ep=\sum_{i\geq1}\epsilon_i^{-1}p_iq^\epsilon_{n-i}$. On the other hand, writing $p_\la=p_1^{m_1}p_2^{m_2}\cdots$ with $m_i=m_i(\la)$, we find $A.p_{\la}=|\la|p_{\la}$; thus $A.q_n^\epsilon=nq_n^\epsilon$.
\end{proof}

\begin{remark}\label{R:generalizingNewton}
In the case of $l(\la)=s=1$ and $\epsilon=(\epsilon_1,\epsilon_2,\dots)=(1,1,\dots)$, $q^\ep_n$ is the Schur function $s_n$ and this gives (\ref{F:NewtonIdentity}), one form of Newton's formula.

In the case of $l(\la)=s=2$ and $\epsilon=(\epsilon_1,\epsilon_2,\dots)=(\al,\al,\dots)$, it gives the \emph{crucial} Lemma 3.8 in \cite{CJ2}, where we consider the following operator $$D(\alpha)=\sum_{i,j\geq 1}h_{-i}h_{-j}h_{i+j}+\sum_{i,j\geq
1}h_{-(i+j)}h_{i}h_{j}+(\alpha-1)\sum_{i\geq 1}ih_{-i}h_i.
$$

Let us show briefly how this operator diagonalizes Jack functions.  We only need to show that $D(\al)$ is self-adjoint and raises the $q^\ep_\la$'s. It is self-adjoint since $h_n^*=h_{-n}$.  From the expression $q^\ep_n=\sum_{\la\vdash n}\al^{-l(\la)}z_\la^{-1}p_\la$, it is not difficult to show that $D(\al).q^\ep_n$ is a scalar multiple of $q^\ep_n$. Then we write $D(\al)=A(\al)+B(\al)$ with $B(\al)=\sum_{i,j\geq
1}h_{-(i+j)}h_{i}h_{j}$. Note that $A(\al)$ is a derivation on $\Lambda(\ep)$ and $B(\al)$ is a second ordered differential operator. When acting on a product $q^\ep_\la=q^\ep_{\la_1}\cdots q^\ep_{\la_s}$ we have
\begin{align*}
&D(\al).q^\ep_\la=\sum_i q^\ep_{\la_1}\cdots(D(\al).q^\ep_{\la_i})\cdots q^\ep_{\la_s}\\
&\qquad\qquad +2\sum_{1\leq k<l\leq s}\sum_{i,j\geq1}h_{-(i+j)}q^\ep_{\la_1}\cdots (h_i.q^\ep_{\la_k})\cdots (h_j.q^\ep_{\la_l})\cdots q^\ep_{\la_s}.
\end{align*}
The first summand is a scalar multiple of $q^\ep_\la$ because $D(\al)$ diagonalizes the $q^\ep_n$'s. Note that $h_i.q^\ep_n=q^\ep_{n-i}$ and by the $s=2$ case of Corollary \ref{R:directgeneralizedNewton}, $\sum_{i,j\geq1}h_{-(i+j)}q^\ep_{\la_k-i}q^\ep_{\la_l-j}$ is of the form $\sum_{(s,t)\geq(\la_k,\la_l)}C_{s,t}q^\ep_{(s,t)}$. Now we see that for each pair $k<l$, the term in the second summand is of
the form  $\sum_{\mu\geq\la}C_{\la\mu}q^\ep_\mu$ by the union property of
partitions. 
Thus $D(\al)$ raises $q^\ep_\la$'s.
\end{remark}

Now we turn to the proof of the theorem.
To simplify the notation, we denote
 \begin{equation}\label{F:defineT}
 T.q_\la=\sum_{i_1,\dots,i_s\geq 1}R_{i_1+\cdots+i_s}q_{\la_1-i_1}\cdots q_{\la_s-i_s},
 \end{equation}
 for a partition $\la=(\la_1,\dots,\la_s)$ of length $s\geq1$.
 Note that we do not allow $0$ parts for $\la$ in the notation $T.q_{\la}$. The reason is that the right side of (\ref{F:defineT}) vanishes if one of the $\la_i=0$.
If the $q_n$ for $n\geq1$ form a set of (algebraic) generators in $\mathcal{A}$, $T$ defines an operator on $\mathcal{A}$. This operator depends on the $R_n$'s.
The following lemma gives an iterative formula for the action of $T$.

\begin{lemma}\label{L:iterationformula}
 Let $(\la_1,\dots,\la_s,a+1)$ be a partition with $a\geq1$. Then we have
\begin{align}\label{F:iteration1}
T.q_{(\la_1,\dots,\la_s,1)}&=T.q_{[\la_1,\dots,\la_{s-1},\la_s+1]}-q_{\la_s}T.q_{(\la_1,\dots,\la_{s-1},1)},\\
\label{F:iteration}
T.q_{(\la_1,\dots,\la_s,a+1)}&=T.q_{[\la_1,\dots,\la_{s-1},\la_s+1,a]}\\\nonumber
                              &+q_aT.q_{[\la_1,\dots,\la_{s-1},\la_s+1]}-q_{\la_s}T.q_{(\la_1,\dots,\la_{s-1},a+1)}.
\end{align}
\end{lemma}

\begin{proof}
The first formula comes directly from the following computations:
\begin{align*}
&T.q_{[\la_1,\dots,\la_{s-1},\la_s+1]}\\
&=\sum_{i_1,\dots,i_s\geq 1}R_{i_1+\cdots+i_s}q_{\la_1-i_1}\cdots q_{\la_{s-1}-i_{s-1}} q_{\la_s+1-i_s}\\
&=\sum_{i_1,\dots,i_{s-1}\geq1,i_s\geq 0}R_{i_1+\cdots+i_s+1}q_{\la_1-i_1}\cdots q_{\la_{s-1}-i_{s-1}}q_{\la_s-i_s}\\
&=\sum_{i_1,\dots,i_{s-1},i_s\geq 1}R_{i_1+\cdots+i_s+1}q_{\la_1-i_1}\cdots q_{\la_s-i_s}\\
&\qquad+\sum_{i_1,\dots,i_{s-1}\geq1}R_{i_1+\cdots+i_{s-1}+1}q_{\la_1-i_1}\cdots q_{\la_{s-1}-i_{s-1}}q_{\la_s}\\
&=\sum_{i_1,\dots,i_{s-1},i_s,j\geq 1}R_{i_1+\cdots+i_s+j}q_{\la_1-i_1}\cdots q_{\la_s-i_s}q_{1-j}\\
&\qquad+q_{\la_s}\sum_{i_1,\dots,i_{s-1},j\geq1}R_{i_1+\cdots+i_{s-1}+j}q_{\la_1-i_1}\cdots q_{\la_{s-1}-i_{s-1}}q_{1-j}\\
&=T.q_{(\la_1,\dots,\la_s,1)}+q_{\la_s}T.q_{(\la_1,\dots,\la_{s-1},1)}.
\end{align*}
Note that for the second last equal sign we use the given condition $q_0=1$.
Directly using this result and then formula (\ref{F:iteration1}), we have
\begin{align*}
T.q_{(\la_1,\dots,\la_s,a+1)}&=T.q_{(\la_1,\dots,\la_{s},a,1)}+q_{a}T.q_{(\la_1,\dots,\la_{s},1)}\\
                              &=T.q_{(\la_1,\dots,\la_{s},a,1)}+q_{a}(T.q_{[\la_1,\dots,\la_{s-1},\la_s+1]}-q_{\la_s}T.q_{(\la_1,\dots,\la_{s-1},1)}).
\end{align*}
Similarly, we have
\begin{align*}
T.q_{[\la_1,\dots,\la_s+1,a]}=T.q_{(\la_1,\dots,\la_{s},a,1)}+q_{\la_s}(T.q_{(\la_1,\dots,\la_{s-1},a+1)}-q_aT.q_{(\la_1,\dots,\la_{s-1},1)}).
\end{align*}
Combining these two formulae, we find formula (\ref{F:iteration}).
\end{proof}

Now we can prove the theorem by using another induction inside the first induction.
\begin{proof}
We use induction on $l(\la)$. The statement is true for $l(\la)=1$. Assume the statement holds for
 any partition $\la$ such that $l(\la)\leq s$. We prove the case of $l(\la)=s+1$ by induction on $\la_{s+1}$.

 We first consider $\la=(\la_1,\dots,\la_s,1)$, thus $\la^+=[\la_1,\dots,\la_{s-1},\la_s+1]>\la$. By formula (\ref{F:iteration1}), we have $T.q_{\la}=T.q_{\la^+}-q_{\la_s}T.q_{(\la_1,\dots,\la_{s-1},1)}$, and the theorem is true for this $\la$ by the induction hypothesis and the
 union property of partitions.

Now we assume that the theorem is true for $\la$ with $l(\la)=s+1$ and $\la_{s+1}=a$ ($a\geq1$). For the partition $\la=(\la_1,\dots,\la_s,a+1)$, we have $\la^+=[\la_1,\dots,\la_{s-1},\la_s+1,a]$, and formula (\ref{F:iteration}) becomes:
$$T.q_{\la}=T.q_{\la^+}+q_aT.q_{[\la_1,\dots,\la_{s-1},\la_s+1]}-q_{\la_s}T.q_{(\la_1,\dots,\la_{s-1},a+1)}.$$
Note that $[\la_1,\dots,\la_{s-1},\la_s+1]\cup(a)=\la^+>\la$, and the term $q_\la$ appears only in $-q_{\la_s}T.q_{(\la_1,\dots,\la_{s-1},a+1)}$. So, by the induction hypothesis and the union property of partitions, the theorem is true in this case.
\end{proof}

\begin{lemma}\label{L:tworowaction}
Under the assumption of  Theorem \ref{T:Newtongeneralization}, for $\la=(m,n)$ of length $2$ we have
$$T.q_{\la}=\sum_{(i,j)>\la}(c_i-c_j)q_{(i,j)}-c_nq_{(m,n)},$$
where we set $c_0=0$ for convenience.
\end{lemma}

\begin{proof}
We prove it by induction on $n$ with the help of Lemma \ref{L:iterationformula}. For $n=1$ we have
\begin{align*}
T.q_\la&=T.q_{n+1}-q_nT.q_1=c_{n+1}q_{n+1}-c_1q_{(n,1)}
\end{align*}
and thus it is true in this case. Assume it is true for $n=a$ ($a\geq1$). Now for $\la=(m,a+1)$ we have
\begin{align*}
T.q_\la&=T.q_{m+1,a}+q_aT.q_{m,1}-q_mT.q_{a,1}\\
       &=\sum_{(i,j)>(m+1,a)}(c_i-c_j)q_{(i,j)}-c_aq_{(m+1,a)}+q_a(c_{m+1}q_{m+1}-c_1q_{(m,1)})\\
       &\qquad\qquad\qquad\qquad\qquad\qquad\qquad\qquad\quad-q_m(c_{a+1}q_{a+1}-c_1q_{(a,1)})\\
       &=\sum_{(i,j)>(m+1,a)}(c_i-c_j)q_{(i,j)}+(c_{m+1}-c_a)q_{(m+1,a)}-c_{a+1}q_{(m,a+1)}\\
       &=\sum_{(i,j)>(m,a+1)}(c_i-c_j)q_{(i,j)}-c_{a+1}q_{(m,a+1)}.
\end{align*}
\end{proof}
 \section{Macdonald symmetric functions}

For two parametric vectors $\eta=(\eta_1,\eta_2,\dots), \tau=(\tau_1,\tau_2,\dots)$ we define the following vertex operator $$X^{\eta\tau}(z)=\exp\Big(\sum_{n\geq1}
\frac{z^{n}}{n}\eta_n h_{-n}\Big)\exp\Big(\sum_{n\geq1}
\frac{z^{-n}}{n}\tau_n h_n\Big)=\sum_n X_{n}^{\eta\tau} z^{-n},$$
which maps $\Lambda(\ep)$ to $\Lambda(\ep)[[z]]$, the set of formal power series of $z$ in $\Lambda(\ep)$.
We want to find the \emph{right} choice of $\epsilon_n, \eta_n$ and $\tau_n$, such that $X_0^{\eta\tau}$ is self-adjoint and  acts on $q^\epsilon_\la$ triangularly.

 For a partition $\mu=(\mu_1,\dots,\mu_s)$, define the operators $h_\mu=h_{\mu_1}\cdots h_{\mu_s}$ and $h_{-\mu}=h_{-\mu_1}\cdots h_{-\mu_s}$. For any series of constants $c=(c_{\mu\nu})_{|\mu|=|\nu|}$, $T_c=\sum_{|\mu|=|\nu|}c_{\mu\nu}h_{-\mu} h_{\nu}$ defines a linear operator on $\Lambda(\ep)$.
 The following lemma can easily be proved.

 \begin{lemma} \label{L:zerooperator} The operator $T_c=0$ if and only if $c=0$, i.e. $T_c.v=0$ for every $v\in \Lambda(\ep)$ if and only if $c_{\mu\nu}=0$ for each pair $(\mu,\nu)$ with $|\mu|=|\nu|$.
 \end{lemma}
The following lemma gives a necessary and sufficient condition for $X_0^{\eta\tau}$ to be self-adjoint.
\begin{lemma}\label{L:conjugate}
Assume  $\eta_1\tau_1\neq0$. We have that $X_0^{\eta\tau}$ is self-adjoint, i.e. $X_0^{\eta\tau}=(X_0^{\eta\tau})^*$  if and only if there is an nonzero constant $c$ such that $\eta_n=\tau_n c^n$.
\end{lemma}

\begin{proof}
It can be shown that $$X_0^{\eta\tau}=\sum_{|\mu|=|\nu|}z_\mu^{-1}\eta_\mu h_{-\mu} z_\nu^{-1}\tau_\nu h_{\nu},$$
where $\eta_\mu=\eta_{\mu_1}\cdots\eta_{\mu_s}$ for a partition $\mu=(\mu_1,\dots,\mu_s)$, and similarly for $\tau_\nu$.
On the other hand we have $(X_0^{\eta\tau})^*=X_0^{\tau\eta}$. By Lemma \ref{L:zerooperator}, $X_0^{\eta\tau}$ is self-adjoint if and only if
$\eta_\mu \tau_\nu=\eta_\nu \tau_\mu$ for each pair $(\mu,\nu)$ with $|\mu|=|\nu|$. Specifically, one needs $\eta_n \tau_1^n=\eta_1^n \tau_n$ for $n\geq1$, and thus $\eta_n=(\eta_1/\tau_1)^n \tau_n$. Note that $\eta_n=\tau_nc^n$ implies $\eta_\la=\tau_\la c^{|\la|}$; the sufficiency follows.
\end{proof}
\begin{lemma}\label{L:creatingpart}
Set $\exp\Big(\sum_{n\geq1}
\frac{z^{n}}{n}\eta_n p_{n}\Big)=\sum_{n\geq0}R_nz^n$. If $\eta_n=\ep_n^{-1}(b^n-1)$, then we have $$\sum_{i\geq1}R_iq^\ep_{n-i}=(b^n-1)q_n^\ep.$$
\end{lemma}
\begin{proof} We compute that
\begin{align*}
&\sum_{n\geq0}(\sum_{i\geq0}R_iq^\ep_{n-i})w^n=\sum_{i\geq0}R_iw^i\cdot\sum_{j\geq0}q^\epsilon_jw^j\\
&=\exp\Big(\sum_{n\geq1}
\frac{w^{n}}{n}\eta_n p_{n}\Big)\exp\Big(\sum_{n\geq1}
\frac{w^n}{n\epsilon_n}p_{n}\Big)=\exp\Big(\sum_{n\geq1}
\frac{w^n}{n\epsilon_n}b^np_{n}\Big)\\
&=\sum_{n\geq0}q^\epsilon_n(bw)^n,   
\end{align*}
which implies the result.
\end{proof}
\begin{theorem}\label{T:target operator}
Fix constants  $a,b, c$, and set
\begin{align} \label{E:relation1}
&\tau_n=1-a^n, \nonumber \\
&\eta_n=(1-a^n)c^n, \\
&\epsilon_n=\frac{b^n-1}{1-a^n}c^{-n}. \nonumber
\end{align}
Then $X^{\eta\tau}_0$ is self-adjoint and raises $q^\epsilon_\la$'s, i.e. $X^{\eta\tau}_0.q^\epsilon_\la=\sum_{\mu\geq\la}c_{\la\mu}q^\epsilon_\mu$.
Moreover $$c_{\la\la}=1+(1-a)\sum_{i=1}^{l(\la)}(b^{\la_i}-1)a^{i-1}.$$
\end{theorem}
\begin{proof} The equations (\ref{E:relation1}) are equivalent to the following relations
\begin{align} \label{E:relation2}
&\tau_n=1-a^n, \nonumber \\
&\eta_n=\epsilon_n^{-1}(b^n-1),\\
&\eta_n=c^{n}\tau_n. \nonumber
\end{align}
By Lemma \ref{L:conjugate}, $X^{\eta\tau}_0$ is self-adjoint.

We also need another operator on $\Lambda(\ep)$ defined by:
\begin{equation}\label{E:qgenerating}
Y^\epsilon(w)=\exp\Big(\sum_{n\geq1}
\frac{w^n}{n\epsilon_n}h_{-n}\Big)=\sum_n Y_{n}^\epsilon w^n.
\end{equation}
Note that the action of $Y^\ep_n$ on $\Lambda(\ep)$ is the multiplication of $q^\ep_n$, i.e. $Y^\ep_{n}.v=q^\ep_{n}\cdot v$.

For a partition $\la=(\la_1,\dots,\la_s)$,
$X^{\eta\tau}_0.q^\epsilon_\la$ is the coefficient of $z^0w_1^{\la_1}\cdot w_s^{\la_s}$ in the following
\begin{align*}
&X^{\eta\tau}(z).Y^\epsilon(w_1)\cdots Y^\epsilon(w_s).1\\
&=\exp\Big(\sum_{n\geq1}
\frac{z^{n}}{n}\eta_n h_{-n}\Big)\exp\Big(\sum_{n\geq1}
\frac{z^{-n}}{n}\tau_n h_n\Big).\prod_{i=1}^s \exp\Big(\sum_{n\geq1}
\frac{w_i^n}{n\epsilon_n}h_{-n}\Big).1\\
&=\exp\Big(\sum_{n\geq1}
\frac{z^{n}}{n}\eta_n h_{-n}\Big)\prod_{i=1}^s \exp\Big(\sum_{n\geq1}
\frac{w_i^n}{n\epsilon_n}h_{-n}\Big)\prod_{i=1}^s \exp\Big(\sum_{n\geq1}
\frac{(w_i/z)^{n}}{n}\tau_n\Big).1\\
&=\sum_{n\geq0}R_nz^n\cdot\sum_{n_1,\dots,n_s\geq0}q_{n_1}w_1^{n_1}\cdots q_{n_s}w_s^{n_s}\cdot\prod_{i=1}^s\frac{1-aw_i/z}{1-w_i/z}.
\end{align*}
Note that $\frac{1-aw_i/z}{1-w_i/z}=\sum_{k\geq0}d_k (w_i/z)^k$, with $d_0=1$ and $d_k=1-a$ for $k\geq1$. We have
\begin{align}\label{E:actionofX0}
X^{\eta\tau}_0.q^\epsilon_\la=\sum_{i_1,\dots,i_s\geq0}R_{i_1+\cdots+i_s}d_{i_1}q^\epsilon_{\la_1-i_1}\cdots d_{i_s}q^\epsilon_{\la_s-i_s}.
\end{align}

With the assistance of Lemma \ref{L:creatingpart}, we can apply Theorem \ref{T:Newtongeneralization} to the right side of expression (\ref{E:actionofX0}) with $c_n=b^n-1$. Note that although these $i_j$'s start from $0$, this will not affect the triangularity of $X^{\eta\tau}_0.q^\epsilon_\la$ by the fact of unions of
partitions (cf. before
Theorem \ref{T:Newtongeneralization}),
and the leading coefficient has the desired form.
\end{proof}
\begin{corollary} For the algebra of symmetric functions $\Lambda(\ep)$, set $\epsilon_n=\frac{b^n-1}{1-a^n}c^{-n}$.
Then for each partition $\la$, there is a unique symmetric function $Q_\la$ such that\\
(1) $Q_\la= \sum_{\mu\geq\la}C_{\la\mu}q^\epsilon_\mu$ with $C_{\la\la}=1,$\\
(2) $Q_\la$ is an eigenvector of  $X^{\eta\tau}_0$.\\
Moreover these $Q_\la$'s give rise to an orthogonal basis of $\Lambda$ and $X^{\eta\tau}_0.Q_\la=c_{\la\la}Q_\la$ with $c_{\la\la}$ given in Theorem \ref{T:target operator}.
\end{corollary}
\begin{proof}
We first prove the second statement by assuming the first statement is true. We first have $X^{\eta\tau}_0.Q_\la=c_{\la\la}Q_\la$ by Theorem \ref{T:target operator}. Note that $c_{\la\la}\neq c_{\mu\mu}$ for $\la\neq\mu$, and also $X^{\eta\tau}_0$ is self-adjoint. We have orthogonality by the following
\begin{align*}
&c_{\la\la}\langle Q_\la,Q_\mu\rangle=\langle X^{\eta\tau}_0.Q_\la,Q_\mu\rangle=\langle Q_\la,X^{\eta\tau}_0.Q_\mu\rangle\\
&=c_{\mu\mu}\langle Q_\la,Q_\mu\rangle.
\end{align*}
 Now we turn to prove the first statement. We need to show that there is a unique family $C_{\la\mu}, \mu\geq\la$ with $C_{\la\la}=1$ such that $\sum_{\mu\geq\la}C_{\la\mu}q_\mu^\epsilon$ is an eigenvector of $X^{\eta\tau}_0$. Thus we need
\begin{align}\label{E:findingcoefficients}
\sum_{\mu\geq\la}c_{\la\la}C_{\la\mu}q_\mu^\epsilon=\sum_{\mu\geq\la}C_{\la\mu}X^{\eta\tau}_0.q_\mu^\epsilon.
\end{align}
Notice here that the eigenvalue has to be $c_{\la\la}$.
Set $X^{\eta\tau}_0.q_\mu^\epsilon=\sum_{\nu\geq\mu} c_{\mu\nu}q_\nu^\epsilon$. Consider the coefficients of $q^\ep_\nu$ for $\nu\geq\la$ in both sides of (\ref{E:findingcoefficients}). For $\nu=\la$, the coefficients in both sides are already equal, and $C_{\la\la}=1$ is fixed. For $\nu>\la$, we need
\begin{align*}
C_{\la\nu}=\frac{\sum_{\nu>\mu\geq\la}C_{\la\mu}c_{\mu\nu}}{c_{\la\la}-c_{\nu\nu}}.
\end{align*}
Starting from $C_{\la\la}=1$, and making induction on the dominance order, each $C_{\la\nu}$ can be found uniquely.
\end{proof}
\begin{remark} Several specializations of $a,b,c$ lead to some interesting results. \\
(A) Set $a=t, b=c=q^{-1}$; then $\epsilon_n=\frac{1-q^n}{1-t^n}$ corresponds to the standard Macdonald case, with the vertex operator and eigenvalue
\begin{align*}
&X^{\eta\tau}(z)=\exp\Big(\sum_{n\geq1}
(1-t^n)q^{-n}\frac{z^{n}h_{-n}}{n} \Big) \exp\Big(\sum_{n\geq1}
(1-t^n) \frac{z^{-n}h_n}{n}\Big),\\
&c_{\la\la}=1+(1-t)\sum_{i=1}^{l(\la)}(q^{-\la_i}-1)t^{i-1}.
\end{align*}
(B) Setting $b=q,~c=t,~a=t^{-1}$, we have $\epsilon_n=\frac{q^n-1}{t^n-1}$. This also gives the Macdonald case, with
 \begin{align*}
 &X^{\eta\tau}(z)=\exp\Big(\sum_{n\geq1}
\frac{h_{-n}z^{n}}{n}(t^{n}-1) \Big)\exp\Big(\sum_{n\geq1}
\frac{h_nz^{-n}}{n}(1-t^{-n})\Big),\\
&c_{\la\la}=1+(1-t^{-1})\sum_{i=1}^{l(\la)}(q^{\la_i}-1)t^{1-i}.
\end{align*}
In this specialization, if we set $E=(X^{\eta\tau}_0-1)/(t-1)$, then $E$ is exactly the limit operator used by Macdonald on p. 321 (Sec. 4, ch. 6) in \cite{M}.\\\\
(C) In (B), we can still set $q=0$ to get an operator diagonalizing Hall-Littlewood functions. The corresponding eigenvalues are different for partitions of different length.\\\\
(D) Setting $b=a^{-1}=c$, we have $\epsilon_n=1$ and this is the Schur case, with
 \begin{align*}
 &X^{\eta\tau}(z)=\exp\Big(\sum_{n\geq1}
\frac{h_{-n}z^{n}}{n}(a^{-n}-1) \Big)\exp\Big(\sum_{n\geq1}
\frac{h_nz^{-n}}{n}(1-a^n)\Big),\\
&c_{\la\la}=1+(1-a)\sum_{i=1}^{l(\la)}(a^{-\la_i}-1)a^{i-1}.
\end{align*}
It can be observed that the eigenvalues associated with different partitions are different.
\end{remark}
\subsection{Example: Jing-J\"ozefiak formula}

We can use Lemma \ref{L:tworowaction} and expression (\ref{E:actionofX0}) to find an explicit formula for two-row functions $Q_{(m,n)}$. Note that now $c_i=b^{i}-1$, and we have
\begin{align*}
X_0^{\eta\tau}.q^\ep_{m,n}&=q^\ep_{(m,n)}+(1-a)[(b^m-1)+(b^n-1)]q^\ep_{(m,n)}\\
           &\qquad+(1-a)^2[\sum_{(i,j)>(m,n)}[(b^i-1)-(b^j-1)]q^\ep_{(i,j)}-(b^n-1)q^\ep_{(m,n)}].
\end{align*}
Setting $X'_0=\frac{X_0^{\eta\tau}-1}{1-a}$, we have
\begin{align*}
X'_0.q^\ep_{(m,n)}&=[(b^m-1)+a(b^n-1)]q^\ep_{(m,n)}+(1-a)\sum_{(i,j)>(m,n)}(b^i-b^j)]q^\ep_{(i,j)}.
\end{align*}
Let $X'_0$ act on both sides of $Q_{(m,n)}=\sum_{i\geq0}g_iq^\ep_{m+i,n-i}$ and consider the coefficients of $q^\ep_{(m+i,n-i)}$ ($g_0=1$). we have:
\begin{align}
\label{F:iterative2row}
&g_i=\frac{(1-a)(b^{m-n+2i}-1)}{(1-b^i)(b^{m-n+i}-a)}\sum_{j<i}g_j.
\end{align}
Starting from $g_0=1$ and iteratively using formula \ref{F:iterative2row}, we can compute $g_1,g_2,\dots$, find the general expression for $g_i$ and prove the following theorem, where both the specializations $a=t, b=q^{-1}$ and $a=t^{-1}, b=q$ lead to the coefficients for two-row Macdonald functions.

Note that the two-row formula has been generalized to multiple rows by Lassalle and Schlosser \cite{LS},
see also \cite{EW}.

\begin{theorem}\cite{JJ} Assume that $Q_{(m,n)}=\sum_{i\geq0}g_iq^\ep_{m+i}q^\ep_{n-i}$, then $g_0=1$ and for $i\geq1$
\begin{equation}\label{E:JJ}
g_i=\frac{(1-ab^0)\cdots(1-ab^{i-1})}{(1-b^1)\cdots(1-b^i)}
\frac{(b^{m-n+1}-1)\cdots(b^{m-n+(i-1)}-1)}{(b^{m-n+1}-a)\cdots(b^{m-n+i}-a)}(b^{m-n+2i}-1).
\end{equation}
\end{theorem}
\begin{proof}
Assume that (\ref{E:JJ}) is true for $i\leq s$. Then these evaluations of $g_1,\dots,g_s$ satisfy equation (\ref{F:iterative2row}) for $i=s$; thus we have
\begin{align*}
&\sum_{j<s+1}g_j=g_s+\sum_{j<s}g_j=g_s+g_s\frac{(b^{m-n+s}-a)(1-b^s)}{(1-a)(b^{m-n+2s}-1)}\\
&=g_s\frac{(b^{m-n+s}-1)(1-ab^s)}{(1-a)(b^{m-n+2s}-1)}.
\end{align*}
By formula \ref{F:iterative2row} we have
\begin{align*}
g_{s+1}&=\frac{(1-a)(b^{m-n+2s+2}-1)}{(1-b^{s+1})(b^{m-n+s+1}-a)}\sum_{j<s+1}g_j\\
       &=\frac{(1-a)(b^{m-n+2s+2}-1)}{(1-b^{s+1})(b^{m-n+s+1}-a)}g_s\frac{(b^{m-n+s}-1)(1-ab^s)}{(1-a)(b^{m-n+2s}-1)}.
\end{align*}
Replacing $g_s$ by its expression, we find that Eq. (\ref{E:JJ}) is true for $i=s+1$.
\end{proof}

\centerline{\bf Acknowledgments}
The first author would like to thank the support of Universit\"at Basel.
He also thanks Professor Michael Schlosser for discussion. The second author
gratefully acknowledges the partial support of Simons Foundation grant 198129,
and NSF grants 1014554 and 1137837 during this work.

\bibliographystyle{amsalpha}

\begin{thebibliography}{ABC}

\bibitem[As]{As} S. Assaf, {\em Dual equivalence graphs I: a combinatorial proof of LLT and Macdonald positivity},
arxiv:1005.3759.

\bibitem[CJ1]{CJ} W. Cai, N. Jing, {\em On vertex operator realizations of Jack functions},
 Jour. Alg. Comb. 32, (2010), 579--595.

\bibitem[CJ2]{CJ2} W. Cai, N. Jing, {\em Applications of
Laplace-Beltrami operator for Jack functions}, European J. Combin. 33 (2012), no. 4, 556--571.

\bibitem[CJ3]{CJ3} W. Cai, N. Jing, {\em Jack vertex operators and realization of Jack functions},
(arXiv:1108.2007).

\bibitem[C]{C} I. Cherednik, Double affine Hecke algebras, Cambridge University Press, Cambridge, 2005.

\bibitem[CG]{CG} H. Coskun, R.~A. Gustafson, {\em Well-poised Macdonald functions $W_{\lambda}$ and Jackson coefficients $\omega_{\lambda}$ on $BC_n$}, Jack, Hall-Littlewood and Macdonald polynomials, pp.127--155,
Contemp. Math., 417, Amer. Math. Soc., Providence, RI, 2006.

\bibitem[EW]{EW} A.-E. Williams, {\em
A formula for N-row Macdonald polynomials}
J. Algebraic Combin. 21 (2005), 111--130.

\bibitem[Ha]{Ha} M. Haiman, {\em Hilbert schemes, polygraphs and the Macdonald positivity conjecture}, J. Amer. Math. Soc. 14 (2001), 941--1006.

\bibitem[HHL]{HHL} J. Haglund, M. Haiman, N. Loehr, {\em A combinatorial formula for Macdonald polynomials}. J. Amer. Math. Soc. 18 (2005), 735--761.

\bibitem[Hg]{Hg} J. Haglund, {\em A combinatorial model for the Macdonald polynomials}, Proc. Natl. Acad. Sci. USA 101 (2004), 16127--16131

\bibitem[GH]{GH} A. Garsia, M. Haiman, {\em A Remarkable q, t-Catalan Sequence and
q-Lagrange Inversion}, Jour. Alg. Comb. 5, (1996), 191--244.

\bibitem[J]{J} N. Jing, {\em Vertex operators and Hall-Littlewood symmetric functions},
Adv. Math. 87 (1991), no. 2, 226--248.   

%
%
\bibitem[JJ]{JJ} N. Jing, T. J\'ozefiak, {\em A formula for two row Macdonald
functions}, Duke Math. J. 67, No. 2 (1992), 377--385.

\bibitem[Ke]{Ke} S.~V. Kerov, {\em Generalized Hall-Littlewood symmetric functions and orthogonal
polynomials}, Representation theory and dynamical systems, 67--94, Adv. Soviet
Math. 9, Amer. Math. Soc., Providence, RI, 1992.

\bibitem[Ka]{Ka} J. Kaneko, {\em $q$-Selberg integrals and Macdonald polynomials}, Ann. Sci. \'Ecole Norm. Sup. (4) 29 (1996), 583--637.

\bibitem[KS]{KS} F. Knop, S. Sahi, {\em A recursion and a combinatorial formula for Jack polynomials},
Invent. Math. 128 (1997), 9--22.

\bibitem[La]{La} A. Lascoux, {\em Yang-Baxter graphs, Jack and Macdonald polynomials},
Ann. Comb. 5 (2001), 397--424.

\bibitem[LLM]{LLM} A. Lascoux, Lapointe, J. Morse, {\em Tableau atoms and a new Macdonald positivity conjecture},
Duke Math. J. 116 (2003), 103--146.  

\bibitem[LS]{LS} M. Lassalle, M.~J. Schlosser, {\em Inversion of the Pieri formula for Macdonald polynomials}, Adv. Math. 202 (2006), 289--325.

\bibitem[M]{M} I. G. Macdonald, {\em Symmetric functions and Hall polynomials},
 2nd ed., With contributions by A. Zelevinsky. Oxford Univ. Press, New York, 1995.

\bibitem[O]{O} A. Okounkov, {\em (Shifted) Macdonald polynomials: q-integral representation and combinatorial formula}, Compositio Math. 112 (1998), 147--182.

\bibitem[R]{R} E. Rains, {\em $BC_n$-symmetric abelian functions}, Duke Math. J. 135 (2006), 99--180.

\bibitem[Sc]{Sc} M. Schlosser, {\em  Macdonald polynomials and multivariable basic hypergeometric series}, SIGMA Symmetry Integrability Geom. Methods Appl. 3 (2007), Paper 056, 30 pp.

\bibitem[S]{S} J. Shiraishi, {\em A family of integral transformations and basic hypergeometric series},
Commun. Math. Phys., 263 (2006), 439--460.

\bibitem[St]{St} R. Stanley, {\em Some combinatorial properties of Jack symmetric functions},
 Adv. Math. 77 (1989), 76--115.

\bibitem[W]{W} S.~O. Warnaar,{\em Bisymmetric functions, Macdonald polynomials and sl3 basic hypergeometric series}, Compos. Math. 144 (2008), 271--303.

\end{thebibliography}

\end{document}